\documentclass[11pt]{article}
\usepackage{srcltx}
\usepackage{xcolor}
\usepackage{graphicx}
\usepackage{amsmath}
\usepackage{amssymb}
\usepackage{theorem}
\usepackage{euscript}
\usepackage{epic,eepic}
\usepackage{pstricks}
\PassOptionsToPackage{normalem}{ulem}
\usepackage{ulem}
\definecolor{labelkey}{rgb}{0,0.08,0.45}
\definecolor{refkey}{rgb}{0,0.6,0.0}
\definecolor{Brown}{rgb}{0.45,0.0,0.05}
\definecolor{dgreen}{rgb}{0.00,0.49,0.00}
\definecolor{dblue}{rgb}{0,0.08,0.75}
\title{\sffamily OUTER APPROXIMATION METHOD FOR CONSTRAINED\\
COMPOSITE FIXED POINT PROBLEMS INVOLVING\\ 
LIPSCHITZ PSEUDO CONTRACTIVE OPERATORS 
\footnote{Contact author: 
Luis M. Brice\~{n}o-Arias, {\ttfamily lbriceno@math.jussieu.fr},
phone: +33 1 4427 8540, fax: +33 1 4427 2724. This work
was supported by the Agence Nationale de la Recherche under
grant ANR-08-BLAN-0294-02.}}
\author{Luis M. Brice\~{n}o-Arias
\\[5mm]
\small UPMC Universit\'e Paris 06\\
\small Laboratoire Jacques-Louis Lions -- UMR 7598\\
\small \'Equipe Combinatoire et Optimisation -- UMR 7090\\
\small 75005 Paris, France (lbriceno@math.jussieu.fr)\\[4mm]
}

\date{}
\RequirePackage[dvipdfm,colorlinks,hyperindex]{hyperref}
\hypersetup{
    linktocpage=true,
    citecolor=dblue,    % color can be changed!
    linkcolor=dgreen    % color can be changed!
}
\newcommand{\Scal}[2]{{\bigg\langle{{#1}\:\bigg |~{#2}}\bigg\rangle}}
\newcommand{\scal}[2]{{\left\langle{{#1}\mid{#2}}\right\rangle}}

\newcommand{\menge}[2]{\big\{{#1}~\big |~{#2}\big\}} 

\newcommand{\bary}{\ensuremath{\widetilde{y}}}
\newcommand{\barr}{\ensuremath{\widetilde{r}}}
\newcommand{\barq}{\ensuremath{\widetilde{q}}}

\newcommand{\HH}{\ensuremath{{\mathcal H}}}

\newcommand{\emp}{\ensuremath{{\varnothing}}}
\newcommand{\lev}[1]{{\ensuremath{{{{\operatorname{lev}}}_{\leq #1}}\,}}}

\newcommand{\Id}{\ensuremath{\operatorname{Id}}\,}
\newcommand{\Fix}{\ensuremath{\operatorname{Fix}}\,}
\newcommand{\RR}{\ensuremath{\mathbb{R}}}

\newcommand{\RPP}{\ensuremath{\left]0,+\infty\right[}}

\newcommand{\RX}{\ensuremath{\left]-\infty,+\infty\right]}}

\newcommand{\NN}{\ensuremath{\mathbb N}}
\newcommand{\ran}{\ensuremath{\operatorname{ran}}}
\newcommand{\zer}{\ensuremath{\operatorname{zer}}}

\newcommand{\gr}{\ensuremath{\operatorname{gra}}}
\newcommand{\conv}{\ensuremath{\operatorname{conv}}}

\newcommand{\pinf}{\ensuremath{{+\infty}}}

\newcommand{\weakly}{\ensuremath{\:\rightharpoonup\:}}

\newcommand{\dom}{\ensuremath{\operatorname{dom}}}

\newcommand{\card}{\ensuremath{\operatorname{card}}}

\newcommand{\inte}{\ensuremath{\operatorname{int}}}

\newtheorem{theorem}{Theorem}[section]
\newtheorem{lemma}[theorem]{Lemma}

\newtheorem{proposition}[theorem]{Proposition}
\theoremstyle{plain}{\theorembodyfont{\rmfamily}%
\newtheorem{assumption}[theorem]{Assumption}}
\theoremstyle{plain}{\theorembodyfont{\rmfamily}%
\newtheorem{algorithm}[theorem]{Algorithm}}
\theoremstyle{plain}{\theorembodyfont{\rmfamily}%
}
\theoremstyle{plain}{\theorembodyfont{\rmfamily}%
\newtheorem{remark}[theorem]{Remark}}
\theoremstyle{plain}{\theorembodyfont{\rmfamily}%
}
\theoremstyle{plain}{\theorembodyfont{\rmfamily}%
\newtheorem{problem}[theorem]{Problem}}
\theoremstyle{plain}{\theorembodyfont{\rmfamily}%
\newtheorem{notation}[theorem]{Notation}}

\numberwithin{equation}{section}
\begin{document}
\maketitle

\begin{abstract}
We propose a method for solving constrained fixed point problems 
involving compositions of Lipschitz pseudo contractive and firmly 
nonexpansive operators in Hilbert spaces. Each iteration of the 
method uses separate evaluations of these operators and an outer 
approximation given by the projection onto a closed half-space 
containing the constraint set. Its convergence is established
and applications to monotone inclusion splitting and constrained 
equilibrium problems are demonstrated.
\end{abstract}

{\small 
\noindent
{\bfseries 2000 Mathematics Subject Classification:}
Primary 65K05;
Secondary 47H05, 47H10, 47J05, 65K15, 90C25.

\noindent
{\bfseries Keywords:}
firmly nonexpansive operator,
fixed point problems,
splitting algorithm,
equilibrium problem,
monotone inclusion,
monotone operator,
pseudo contractive operator.
}

\newpage
\section{Introduction}
The problem under consideration in this paper is the following.
\begin{problem}
\label{prob:1}
Let $\HH$ be a real Hilbert space, fix 
$\varepsilon\in\left]0,1\right[$, and let $(\beta_n)_{n\in\NN}$ 
be a sequence in $\left]0,1-\varepsilon\right]$.
For every $n\in\NN$, 
let $T_n\colon\HH\to\HH$ be a firmly nonexpansive operator, let
$R_n\colon\dom R_n\subset\HH\to\HH$ be a pseudo contraction 
such that $(\Id-R_n)$ is $\beta_n$--Lipschitz continuous, and 
let $S$ be a closed convex subset of $\HH$. The problem is to
\begin{equation}
\label{e:prob1}
\text{find}\quad x\in S\quad\text{such that}\quad(\forall n\in\NN)
\quad T_nR_nx=x.
\end{equation}
The set of solutions to \eqref{e:prob1} is denoted by $Z$.
\end{problem}

As will be seen subsequently, this formulation models a broad range
of problems in nonlinear analysis. Methods can be found in the 
literature to solve Problem~\ref{prob:1} in special cases. 
Thus, when $S=\HH$, $R_n\equiv\Id$, and $Z\neq\emp$, algorithms 
can be found in \cite{Bro67b,Opti04}, and when $S=\HH$,
$T_n\equiv\Id$, and $R_n\equiv R$, where $R$ is a Lipschitz pseudo
contraction from a convex set $C$ into itself, methods can be found 
in \cite{Bruc74,Ishi74,Schu93,Zhou08}.
Since the composition between a firmly nonexpansive operator and a 
Lipschitz pseudo contraction is not a pseudo contraction in general, 
Problem~\ref{prob:1} can not be solved by the methods mentioned 
above. The purpose of the present paper is to provide an algorithm for 
solving Problem~\ref{prob:1}. It involves four elementary steps
at each iteration $n$: the first three steps are successive 
computations of operators $R_n$, $T_n$, and $R_n$, and the last
step is an outer approximation of the constraint. The latter is 
given by the projection onto a half-space containing $S$.
In Section~\ref{sec:2} we propose our algorithm and we prove its 
weak convergence to a solution to Problem~\ref{prob:1}. In 
Section~\ref{sec:3} we study an application to monotone inclusions 
under convex constraints, and obtain an extension of a result of 
\cite{Tsen00}. Finally, in Section~\ref{sec:4}, we study an 
application to equilibrium problems with convex constraints. 

\begin{notation}
Throughout this paper $\HH$ denotes a real Hilbert space, 
$\scal{\cdot}{\cdot}$ denotes its scalar product, and 
$\|\cdot\|$ denotes the associated norm.
For a single-valued operator $R\colon\dom R\subset\HH\to\HH$, 
the set of fixed points is $\Fix R=\menge{x\in\HH}{x=Rx}$,
$R$ is $\chi$--Lipschitz continuous for some $\chi\in\RPP$, if
it satisfies, for every $x$ and $y$ in $\HH$, $\|Rx-Ry\|\leq
\chi\|x-y\|$, $R$ is pseudo contractive if it satisfies
\begin{equation}
\label{e:pseudc}
(\forall x\in\dom R)(\forall y\in\dom R)\quad
\|Rx-Ry\|^2\leq\|x-y\|^2+\|(\Id-R)x-(\Id-R)y\|^2,
\end{equation}
$R$ is firmly nonexpansive if it satisfies
\begin{equation}
\label{e:firmnonexp}
(\forall x\in\dom R)(\forall y\in\dom R)\quad
\|Rx-Ry\|^2\leq\|x-y\|^2-\|(\Id-R)x-(\Id-R)y\|^2,
\end{equation}
or equivalently,
\begin{equation}
\label{e:firmnonexp2}
(\forall x\in\dom R)(\forall y\in\dom R)\quad
\scal{x-y}{Rx-Ry}\geq\|Rx-Ry\|^2,
\end{equation}
and $R$ is $\chi$--cocoercive if $\chi R$ is firmly nonexpansive.
\end{notation}

\section{Algorithm and convergence}
\label{sec:2}

At each iteration $n\in\NN$, our method for solving Problem~\ref{prob:1}
involves an outer approximation to $S$ and separate computations 
of the operators $T_n$ and $R_n$. Each approximation is computed by the 
projection onto a closed affine half-space containing $S$, and errors
on the computation of the operators are modeled by the sequences
$(a_n)_{n\in\NN}$, $(b_n)_{n\in\NN}$, and $(c_n)_{n\in\NN}$.

\begin{algorithm}
\label{algo:1}
Let $(T_n)_{n\in\NN}$, $(R_n)_{n\in\NN}$, and $S$ be as in 
Problem~\ref{prob:1}. For every $n\in\NN$, let $Q_n\colon\HH\to\HH$ 
be the projector operator onto a closed affine half-space containing 
$S$, let $(a_n)_{n\in\NN}$, $(b_n)_{n\in\NN}$, and
$(c_n)_{n\in\NN}$ be sequences in $\HH$ such that 
$\sum_{n\in\NN}\|a_n\|<\pinf$, $\sum_{n\in\NN}\|b_n\|<\pinf$,
and $\sum_{n\in\NN}\|c_n\|<\pinf$. Moreover, let 
$\varepsilon\in\left]0,1\right[$, let 
$(\lambda_n)_{n\in\NN}$ be a sequence in 
$\left[\varepsilon,1\right]$, let $x_0\in\dom R_0$, and consider the 
following routine.
\begin{align}
\label{e:iter1}
(\forall n\in\NN)\quad
&\left 
\lfloor 
\begin{array}{l}
y_n=R_nx_n+a_n\\
q_n=T_ny_n+b_n\\
\text{If }q_n\notin\dom R_n\:\text{stop}.\\
\text{\rm Else}\\
\left 
\lfloor 
\begin{array}{l}
r_n=R_nq_n+c_n\\
z_n=x_n-y_n+r_n\\
x_{n+1}=x_n+\lambda_n\big(Q_nz_n-x_n\big)\\
\end{array}
\right.\\
\text{\rm If}\: x_{n+1}\notin\dom R_{n+1}\:\text{stop}.\\
\text{\rm Else } n=n+1.
\end{array}
\right.
\end{align}
\end{algorithm}
Our main result is the following.
\begin{theorem}
\label{t:0}
Suppose that $Z\neq\emp$ in Problem~\ref{prob:1} and
that Algorithm~\ref{algo:1} generates infinite orbits 
$(x_n)_{n\in\NN}$ and $(z_n)_{n\in\NN}$ such that 
\begin{equation}
\label{e:condconv}
(\forall x\in\HH)\quad
\begin{cases}
x_{k_n}\weakly x\\
x_n-T_nR_nx_n\to 0\\
z_n-x_n\to0\\
z_n-Q_nz_n\to 0
\end{cases}\qquad
\Rightarrow \quad x\in Z.
\end{equation}
Then $(x_n)_{n\in\NN}$ converges 
weakly to a solution to Problem~\ref{prob:1}.
\end{theorem}

\begin{proof}
Set 
\begin{equation}
\label{e:proof21}
(\forall n\in\NN)\quad
\bary_n=R_nx_n,\quad
\barq_n=T_n\bary_n,\quad\text{and}\quad
\barr_n=R_n\barq_n,
\end{equation}
fix $z\in Z$, and let $n\in\NN$. Note that, since $z\in S$, we have 
\begin{equation}
\label{e:px}
z=P_Sz=Q_nz=T_nR_nz=R_nz+(\Id-R_n)T_nR_nz.
\end{equation}
In addition, it follows from \cite[Theorem~1]{BroP67} that 
$(\Id-R_n)$ is monotone, which yields 
$\scal{(\Id-R_n)\barq_n-(\Id-R_n)z}{\barq_n-z}\geq 0$.
Therefore, we deduce from \eqref{e:px}, \eqref{e:proof21}, and the 
firm nonexpansivity of $T_n$ that
\begin{align}
\label{e:scal1}
2\scal{\barq_n-z}{(\Id-R_n)x_n-(\Id-R_n)\barq_n}
&=-2\scal{\barq_n-z}
{(\Id-R_n)\barq_n-(\Id-R_n)z)}
\nonumber\\
&\hspace{1cm}+2\scal{\barq_n-z}{x_n-z}
-2\scal{\barq_n-z}{R_nx_n-R_nz}\nonumber\\
&\leq2\scal{\barq_n-z}{x_n-z}
-2\scal{T_n\bary_n-T_nR_nz}{\bary_n-R_nz}\nonumber\\
&\leq2\scal{\barq_n-z}{x_n-z}-
2\|T_n\bary_n-T_nR_nz\|^2\nonumber\\
&=\big(2\scal{\barq_n-z}{x_n-z}-\|\barq_n-z\|^2\big)
-\|\barq_n-z\|^2\nonumber\\
&\leq\|x_n-z\|^2-\|\barq_n-x_n\|^2
-\|\barq_n-z\|^2.
\end{align}
Hence, since $\sup_{k\in\NN}\beta_k^2\leq(1-\varepsilon)^2
\leq1-\varepsilon$, it follows from \eqref{e:proof21} and the 
$\beta_n$--Lipschitz continuity of $(\Id-R_n)$ that
\begin{align}
\label{e:eq000}
\|x_n-\bary_n+\barr_n-z\|^2
&=\|\barq_n-z+(x_n-\bary_n)-(\barq_n-\barr_n)\|^2\nonumber\\
&=\|\barq_n-z+(\Id-R_n)x_n-(\Id-R_n)\barq_n\|^2\nonumber\\
&=\|\barq_n-z\|^2+
\|(\Id-R_n)x_n-(\Id-R_n)\barq_n\|^2\nonumber\\
&\hspace{3.7cm}+2\scal{\barq_n-z}{(\Id-R_n)x_n
-(\Id-R_n)\barq_n}\nonumber\\
&\leq\|\barq_n-z\|^2+\beta_n^{2}\|\barq_n-x_n\|^2\nonumber\\
&\hspace{3.7cm}+2\scal{\barq_n-z}{(\Id-R_n)x_n
-(\Id-R_n)\barq_n}\nonumber\\
&\leq\|x_n-z\|^2-(1-\beta_n^{2})\|\barq_n-x_n\|^2\nonumber\\
&\leq\|x_n-z\|^2-\varepsilon\|\barq_n-x_n\|^2,
\end{align}
which yields
\begin{equation}
\label{e:1212}
\|x_n-\bary_n+\barr_n-z\|\leq\|x_n-z\|.
\end{equation}
We also derive from \eqref{e:iter1} and \eqref{e:proof21}
the following inequalities. First, $\|y_n-\bary_n\|=\|a_n\|$, 
and since $T_n$ is nonexpansive, we obtain
\begin{equation}
\label{e:ine2}
\|q_n-\barq_n\|=\|T_ny_n+b_n-T_n\bary_n\|
\leq\|\bary_n-y_n\|+\|b_n\|=\|a_n\|+\|b_n\|.
\end{equation}
In turn, it follows from the $\beta_n$--Lipschitz continuity
of $(\Id-R_n)$ that
\begin{align}
\label{e:ine3}
\|r_n-\barr_n\|&=\|R_nq_n+c_n-R_n\barq_n\|\nonumber\\
&\leq\|(\Id-R_n)\barq_n-(\Id-R_n)q_n\|+\|q_n-\barq_n\|
+\|c_n\|\nonumber\\
&\leq(1+\beta_n)\|q_n-\barq_n\|+\|c_n\|\nonumber\\
&\leq 2(\|a_n\|+\|b_n\|)+\|c_n\|.
\end{align}
Altogether, if we set 
\begin{equation}
\label{e:errr}
e_n=\bary_n-y_n+r_n-\barr_n,
\end{equation}
we have
\begin{equation}
\|e_n\|=\|\bary_n-y_n+r_n-\barr_n\|\leq\|y_n-\bary_n\|
+\|r_n-\barr_n\|\leq3\|a_n\|+2\|b_n\|+\|c_n\|,
\end{equation}
and therefore $\sum_{k\in\NN}\|e_k\|<\pinf$.
Hence, from \eqref{e:iter1}, \eqref{e:px}, the nonexpansivity
of $Q_n$, and \eqref{e:1212} we get
\begin{align}
\label{e:y_netc}
\|x_{n+1}-z\|&=\|(1-\lambda_n)(x_n-z)+\lambda_n(Q_nz_n-Q_nz)\|
\nonumber\\
&\leq (1-\lambda_n)\|x_n-z\|+
\lambda_n\|Q_nz_n-Q_nz\|\nonumber\\
&\leq (1-\lambda_n)\|x_n-z\|+\lambda_n\|z_n-z\|\nonumber\\
&\leq(1-\lambda_n)\|x_n-z\|+\lambda_n\big(\|x_n-\bary_n
+\barr_n-z\|+\|e_n\|\big)\nonumber\\
&\leq\|x_n-z\|+\|e_n\|, 
\end{align}
and we conclude from \cite[Lemma~3.1]{Comb01} that
\begin{equation}
\label{e:sumfejer}
\xi=\sup_{k\in\NN}\|x_k-z\|<\pinf.
\end{equation}
Thus, from the convexity of $\|\cdot\|^2$, the firm nonexpansivity
of $Q_n$, \eqref{e:px}, \eqref{e:iter1}, 
\eqref{e:errr}, \eqref{e:eq000}, and 
\eqref{e:1212} we have
\begin{align}
\|x_{n+1}-z\|^2
&\leq 
(1-\lambda_n)\|x_n-z\|^2+\lambda_n\|Q_nz_n-Q_nz\|^2\nonumber\\
&\leq (1-\lambda_n)\|x_n-z\|^2+\lambda_n\big(\|z_n-z\|^2
-\|z_n-Q_nz_n\|^2\big)\nonumber\\
&\leq (1-\lambda_n)\|x_n-z\|^2+\lambda_n
\big(\|x_n-\bary_n+\barr_n-z\|^2+\|e_n\|^2
\nonumber\\
&\hspace{3cm}+2\,\|x_n-\bary_n+\barr_n-z\|\,\|e_n\|
-\|z_n-Q_nz_n\|^2\big)\nonumber\\
&\leq(1-\lambda_n)\|x_n-z\|^2+\lambda_n\big(\|x_n-z\|^2-
\varepsilon\|\barq_n-x_n\|^2\nonumber\\
&\hspace{5cm}+\|e_n\|^2+2\,\|x_n-z\|\,\|e_n\|
-\|z_n-Q_nz_n\|^2\big)\nonumber\\
&\leq \|x_n-z\|^2-\varepsilon^2\|\barq_n-x_n\|^2
-\varepsilon\,\|z_n-Q_nz_n\|^2+\eta_n,
\end{align}
where $\eta_n=\|e_{n}\|^2+2\xi\|e_{n}\|$ satisfies 
$\sum_{k\in\NN}\eta_k<\pinf$. Hence, from \cite[Lemma~3.1]{Comb01} 
we deduce that 
\begin{equation}
\label{e:sumabil}
\sum_{k\in\NN}\|T_kR_kx_k-x_k\|^2=\sum_{k\in\NN}\|\barq_k-x_k\|^2
<\pinf\quad\text{and}\quad
\sum_{k\in\NN}\|z_k-Q_kz_k\|^2<\pinf,
\end{equation}
and therefore $T_nR_nx_n-x_n=\barq_n-x_n\to0$ and $z_n-Q_nz_n\to0$.
Thus, it follows from \eqref{e:iter1} and the nonexpansivity of
$T_n$ that 
\begin{align}
\|z_n-x_n\|&=\|r_n-y_n\|\nonumber\\
&=\|\barr_n-\bary_n+e_n\|\nonumber\\ 
&\leq\|T_n\barq_n-T_nx_n\|+\|e_n\|\nonumber\\
&\leq \|\barq_n-x_n\|+\|e_n\|\nonumber\\
&\to 0.
\end{align}
Altogether, since \eqref{e:condconv} asserts that all the
weak limits of the sequence $(x_k)_{k\in\NN}$ are in $Z$,
the result follows from \cite[Theorem~3.8]{Comb01}.
\end{proof}

\section{Monotone inclusions with convex constraints}
\label{sec:3}
We consider the problem
\begin{equation}
\label{e:incmonmain}
\text{find}\quad x\in S\quad\text{such that}\quad 0\in Ax+Bx, 
\end{equation}
where $A\colon\HH\to 2^{\HH}$ and $B\colon\dom B\subset\HH\to\HH$
are maximally monotone, and $S\subset\HH$ is nonempty, closed, and
convex. When $B$ is cocoercive, $\dom B=\HH$, and $S=\HH$, 
\eqref{e:incmonmain} models wide variety of problems in nonlinear 
analysis, and it can be solved by the forward-backward splitting 
method \cite{Sico10,Smms05,Merc80,Tsen90,Tsen91}. However, in several 
applications these assumptions are very restrictive. If the 
cocoercivity of $B$ is relaxed to Lipschitz continuity, 
\eqref{e:incmonmain} can be solved by the modified forward-backward 
splitting in \cite{Tsen00}. We propose an extension of this method 
for solving \eqref{e:incmonmain} with a finite number of convex 
constraints. In addition, our method allows for errors in the 
computations of the operators involved.

\begin{notation}
For a set-valued operator $A\colon\HH\to 2^{\HH}$, 
$\dom A=\menge{x\in\HH}{Ax\neq\emp}$ is the domain of $A$,
$\zer A=\menge{x\in\HH}{0\in Ax}$ is its set of zeros, 
and $\gr A=\menge{(x,u)\in\HH\times\HH}{u\in Ax}$ 
is its graph. The operator $A$ is monotone if it satisfies,
for every $(x,u)$ and $(y,v)$ in $\gr A$, $\scal{x-y}{u-v}\geq0$,
and it is maximally monotone if its graph is not properly
contained in the graph of any other monotone operator acting on $\HH$.
In this case, the resolvent of $A$, $J_A=(\Id+A)^{-1}$, is well 
defined, single-valued, $\dom J_A=\HH$, and it is firmly 
nonexpansive. For every $\alpha\in\RR$, the lower level set at height
$\alpha$ of a function $f\colon\HH\to\RX$ is the closed convex set 
$\lev{\alpha}f=\menge{x\in\HH}{f(x)\leq\alpha}$ and the 
subdifferential of $f$ is the operator
\begin{equation}
\partial f\colon\HH\to 2^{\HH}\colon x\mapsto \menge{u\in\HH}
{(\forall y\in\HH)\quad\scal{y-x}{u}+f(x)\leq f(y)}.
\end{equation}
Now let $C$ be a nonempty subset of $\HH$. 
Then $\inte C$ is the interior of $C$ and if $C$ is nonempty, convex, 
and closed, then $P_C$ denotes the projector operator onto $C$, 
which, for every $x\in\HH$ satisfies 
$\|x-P_Cx\|=\min_{y\in C}\|x-y\|=d_C(x)$, 
where $d_C$ denotes the distance function of $C$.
For further background in monotone operator theory and convex 
analysis see \cite{Livre1}.
\end{notation}

\begin{problem}
\label{prob:tseng}
Let $A\colon\HH\to 2^{\HH}$ and 
$B\colon\dom B\subset\HH\to\HH$ be two maximally monotone operators 
such that $\dom A\subset\dom B$ and suppose that $A+B$ is maximally 
monotone (see \cite[Corollary~24.4]{Livre1} for some sufficient 
conditions). For every $i\in\{1,\ldots,m\}$, let
$f_i\colon\HH\to\RR$ be lower semicontinuous and convex, denote by 
$S=\lev{0}f_1\cap\cdots\cap\lev{0}f_m\neq\emp$, and
assume that $S\subset\dom B$ and that $B$ is $\chi$-Lipschitz 
continuous on $S\cup\dom A$, for some $\chi\in\RPP$.
The problem is to
\begin{equation}
\text{find}\quad x\in\HH\quad\text{such that}\quad
\begin{cases}
x\in \zer(A+B)\\
\:f_1(x)\hspace{.05cm}\leq 0\\
\hspace{1cm}\vdots\\
f_m(x)\leq 0.
\end{cases}
\end{equation}
\end{problem}

Problem~\ref{prob:tseng} models various applications to economics, 
traffic theory, Nash equilibrium problems, and network equilibrium 
problems among others (see \cite{Bert82,Dafe92,Facc03} and the 
references therein). 

In the particular case when $m=1$, $f_1=d_C$, and $C\subset\HH$ is 
a nonempty closed convex set, an algorithm for solving 
Problem~\ref{prob:tseng} is proposed in \cite{Tsen00}, 
without considering errors in the computations and assuming that 
$P_C$ is easily computable (see also \cite{SolS99} for an approach 
using enlargements of maximally monotone operators). However, 
since $P_S$ is not computable in general, Problem~\ref{prob:tseng}
can not be solved by this method. We propose an algorithm for 
solving Problem~\ref{prob:tseng} in which the constraints 
$f_1\leq0$,$\ldots$, $f_m\leq0$ are activated independently
and linearized, and where errors in the computation of the operators 
involved are permitted. For the implementation of this method 
we use the subgradient projector with respect to $f\in\Gamma_0(\HH)$, 
which is defined by
\begin{equation}
\label{e:projsub}
G\colon\HH\to\HH\colon x\mapsto 
\begin{cases}
x-\displaystyle{\frac{f(x)}{\|u\|^2}}u,\quad&\text{if}
\:f(x)>0;\\
x,&\text{otherwise,}
\end{cases}
\end{equation}
where $u\in\partial f(x)$, 
and the function $\operatorname{i}\colon\NN\to\{1,\ldots,m\}
\colon n\mapsto1+\operatorname{rem}(n-1,m)$,
where $\operatorname{rem}(\cdot,m)$ is the 
remainder function of division by $m$.
\begin{algorithm}
\label{algo:tseng}
For every $i\in\{1,\ldots,m\}$, denote by
$G_{i}\colon\HH\to\HH$ the subgradient projector with respect to 
$f_i$. Let $(e_{1,n})_{n\in\NN}$, 
$(e_{2,n})_{n\in\NN}$, and $(e_{3,n})_{n\in\NN}$ be sequences in 
$\HH$ such that $\sum_{n\in\NN}\|e_{1,n}\|<\pinf$, 
$\sum_{n\in\NN}\|e_{2,n}\|<\pinf$, and 
$\sum_{n\in\NN}\|e_{3,n}\|<\pinf$. Let 
$\varepsilon\in\left]0,1/(\chi+1)\right[$, let $(\gamma_n)_{n\in\NN}$
be a sequence in $\left[\varepsilon,(1-\varepsilon)/\chi\right]$,
let $x_0\in\dom B$, and let $(x_n)_{n\in\NN}$ be the sequence 
generated by the following routine.
\begin{align}
\label{e:tseng1}
(\forall n\in\NN)\quad
&\left\lfloor 
\begin{array}{l}
y_n=x_n-\gamma_n(Bx_n+e_{1,n})\\
q_n=J_{\gamma_nA}(y_n+e_{2,n})\\
r_n=q_n-\gamma_n(Bq_n+e_{3,n})\\
z_n=x_n-y_n+r_n\\
x_{n+1}=G_{\operatorname{i}(n)}\,z_n.
\end{array}
\right.
\end{align}
\end{algorithm}
\begin{remark}
In Algorithm~\ref{algo:tseng}, the sequences $(e_{1,n})_{n\in\NN}$ 
and $(e_{3,n})_{n\in\NN}$ represent errors in the computation 
of the operator $B$. In addition, we suppose that the resolvents
$(J_{\gamma_nA})_{n\in\NN}$ can be computed approximatively by
solving, for every $n\in\NN$, the perturbed inclusion
\begin{equation}
\text{find}\quad q\in\HH\quad\text{such that}\quad 
y_n-q+e_{2,n}\in\gamma_nAq.
\end{equation}
\end{remark}

\begin{proposition}
\label{t:Tseng}
Suppose that 
\begin{equation}
\bigcup_{i=1}^m\ran G_i\subset\dom B\quad\text{and}\quad
S\cap\zer(A+B)\neq\emp.
\end{equation}
Then 
Algorithm~\ref{algo:tseng} generates an infinite
orbit $(x_n)_{n\in\NN}$ which converges weakly to a solution to 
Problem~\ref{prob:tseng}.
\end{proposition}
\begin{proof}
Set 
\begin{equation}
\label{e:oper2}
(\forall n\in\NN)\quad\beta_n=\gamma_n\chi,\quad
T_n=J_{\gamma_n A},\quad\text{and}\quad
R_n=\Id-\gamma_n B.
\end{equation}
Note that $(\beta_n)_{n\in\NN}$ is a sequence in 
$\left]0,1-\varepsilon\right]$ and, for every $n\in\NN$, 
$T_n$ is firmly nonexpansive and $\Id-R_n=\gamma_nB$ is 
$\beta_n$--Lipschitz-continuous and monotone. Hence, it follows from 
\cite[Theorem~1]{BroP67} that the operators $(R_n)_{n\in\NN}$ are 
pseudo contractive. 
In addition, note that
$x\in \zer(A+B)\,\, \Leftrightarrow
\:(\forall n\in\NN)\quad x-\gamma_nBx\in x+\gamma_nAx\;
\Leftrightarrow
\;(\forall n\in\NN)\quad x\in\Fix T_nR_n$.
Altogether, we deduce that Problem~\ref{prob:tseng} is
a particular case of Problem~\ref{prob:1} and 
\begin{equation}
Z=S\cap\bigcap_{n\in\NN}\Fix T_nR_n=S\cap\zer(A+B)\neq\emp.
\end{equation}
Now let us prove that Algorithm~\ref{algo:tseng} is a particular 
case of Algorithm~\ref{algo:1}. Set 
\begin{equation}
\label{e:erreure}
(\forall n\in\NN)\quad
\begin{cases}
a_n=-\gamma_ne_{1,n}\\
b_n=J_{\gamma_n A}(y_n+e_{2,n})-J_{\gamma_n A}y_n\\
c_n=-\gamma_ne_{3,n}
\end{cases}
\quad\text{and}\quad
\begin{cases}
\lambda_n=1\\
Q_n=G_{\operatorname{i}(n)}.
\end{cases}
\end{equation}
Then, since $\sup_{n\in\NN}\gamma_n<\chi^{-1}$, we have 
$\sum_{n\in\NN}\|a_n\|<\pinf$ and $\sum_{n\in\NN}\|c_n\|<\pinf$. 
Moreover, from the nonexpansivity of $(J_{\gamma_nA})_{n\in\NN}$, we 
deduce that $\sum_{n\in\NN}\|b_n\|<\pinf$, and, for every $x\in\HH$ and 
$n\in\NN$, $Q_nx$ is the projection onto the closed affine half-space
$\menge{y\in\HH}{\scal{x-y}{u}\geq f_{\operatorname{i}(n)}(x)}$, for
some $u\in\partial f_{\operatorname{i}(n)}(x)$, 
which contains $\lev{0}f_{\operatorname{i}(n)}\supset S$. On the other hand,
$x_0\in\dom B$ and since, for every $i\in\{1,\ldots, m\}$, 
$\ran G_i\subset\dom B$, it follows from \eqref{e:tseng1} that, 
for every $n\in\NN\setminus\{0\}$, $x_n\in\dom B$. In addition, 
$q_n=J_{\gamma_nA}(y_n+e_{2,n})\in\dom A\subset\dom B$. Altogether,
from \eqref{e:oper2} and \eqref{e:erreure}, we deduce that
Algorithm~\ref{algo:tseng} is a particular case of 
Algorithm~\ref{algo:1} and that it generates an infinite orbit 
$(x_n)_{n\in\NN}$.

Let us prove that condition \eqref{e:condconv} holds. Suppose
that $x_{k_n}\weakly x$, $x_n-T_nR_nx_n\to 0$, $z_n-x_n\to0$, 
$z_n-Q_nz_n\to0$, and, for every $n\in\NN$, denote by $p_n=T_nR_nx_n$. 
Hence, $p_{k_n}\weakly x$ and from \eqref{e:oper2} we obtain,
for every $n\in\NN$,
\begin{align}
\label{e:eqclosed}
p_n=T_nR_nx_n
\quad&\Leftrightarrow\quad x_n-\gamma_nBx_n\in p_n+\gamma_nAp_n
\nonumber\\
\quad&\Leftrightarrow\quad \frac{1}{\gamma_n}(x_n-p_n)-Bx_n
\in Ap_n\nonumber\\
&\Leftrightarrow\quad \frac{1}{\gamma_n}(x_n-p_n)+Bp_n-Bx_n
\in (A+B)p_n.
\end{align}
Now, since $A+B$ is maximally monotone, from 
\cite[Proposition~20.33]{Livre1}, its graph is
sequentially weak-strong closed. Therefore, since 
$x_{k_n}-p_{k_n}\to0$, $\|Bp_{k_n}-Bx_{k_n}\|\leq\chi\|x_{k_n}-
p_{k_n}\|\to 0$, $\gamma_{k_n}\geq \varepsilon>0$, 
$p_{k_n}\weakly x$, we conclude from \eqref{e:eqclosed} that 
$x\in\zer(A+B)$. Now let us prove that, for every 
$i\in\{1,\ldots,m\}$, $f_i(x)\leq 0$. Fix $i\in\{1,\ldots,m\}$ and, 
for every $n\in\NN$, let $j_n\in\NN$ such that $k_n\leq j_n\leq k_n+m$ 
and $\operatorname{i}(j_n)=i$. We deduce from $z_n-x_n\to0$ 
and $z_n-Q_nz_n\to0$ that, for every $n\in\NN$, $\|x_{n+1}-x_n\|=
\|Q_nz_n-x_n\|\leq\|Q_nz_n-z_n\|+\|z_n-x_n\|\to0$. Therefore,
\begin{equation}
(\forall n\in\NN)\quad
\|x_{j_n}-x_{k_n}\|\leq\sum_{\ell=k_n}^{j_n-1}\|x_{\ell+1}-x_{\ell}\|
\leq m\underset{k_n\leq\ell\leq k_n+m}{\max}\|x_{\ell+1}-x_{\ell}\|
\to 0
\end{equation}
and hence it follows from $z_{j_n}-x_{j_n}\to0$ and $x_{k_n}\weakly x$
that $z_{j_n}\weakly x$. 
Note that, from \eqref{e:erreure} and \eqref{e:projsub} we have,
for some $u_{j_n}\in\partial f_i(z_{j_n})$,
\begin{equation}
(\forall n\in\NN)\quad
Q_{j_n}z_{j_n}-z_{j_n}=
\begin{cases}
-\displaystyle{\frac{f_i(z_{j_n})}{\|u_{j_n}\|^2}u_{j_n}},
\quad&\text{if}\:f_i(z_{j_n})>0;\\
0,&\text{otherwise},
\end{cases}
\end{equation}
and, since $\|Q_{j_n}z_{j_n}-z_{j_n}\|\to0$, we deduce that
$\max\{0,f_i(z_{j_n})\}\to 0$. Thus, it follows from 
$z_{j_n}\weakly x$ that $f_i(x)\leq \varliminf\,f_i(z_{j_n})\leq 
\varliminf\,\max\{0,f_i(z_{j_n})\}= 0$,
and hence $x\in \lev{0}f_i$. We conclude that $x\in Z$ and 
the result follows from Theorem~\ref{t:0}.
\end{proof}

\begin{remark}
Let us consider the particular case of Theorem~\ref{t:Tseng}
obtained when $e_{1,n}\equiv e_{2,n}\equiv 
e_{3,n}\equiv0$, $m=1$, and $f_1=d_C$, where $C\subset\HH$ is a nonempty 
closed convex set. Then, since $G_1=P_C$, Algorithm~\ref{algo:tseng} 
reduces to the method proposed in \cite{Tsen00}. Moreover, since
$S=C$, note that the assumption $\ran G_1\subset\dom B$ is equivalent 
to $S\subset\dom B$, which was already assumed in Problem~\ref{prob:tseng}.
\end{remark}

\section{Equilibrium problems with convex constraints}
\label{sec:4}
We consider the problem
\begin{equation}
\label{e:eqprobl}
\text{find}\quad x\in C\quad\text{such that}\quad 
(\forall y\in C)\quad F(x,y)\geq0, 
\end{equation}
where $C$ and $F$ satisfy the following assumption.
\begin{assumption}
\label{a:1}
$C$ is a nonempty closed convex subset of $\HH$ and 
$F\colon C^2\to\RR$ satisfies the following.
\begin{enumerate}
\item 
\label{a:1i}
$(\forall x\in C)\quad F(x,x)=0$.
\item 
\label{a:1ii}
$(\forall (x,y)\in C^2)\quad F(x,y)+F(y,x)\leq0$.
\item 
\label{a:1iii} 
For every $x$ in $C$, $F(x,\cdot)\colon C\to\RR$ is lower 
semicontinuous and convex.
\item 
\label{a:1iv} 
$(\forall (x,y,z)\in C^3)\quad \underset{\varepsilon\to 0^+}
{\varlimsup}\:F((1-\varepsilon)x+\varepsilon z,y)\leq F(x,y)$.
\end{enumerate}
\end{assumption}
We are interested in solving a more general problem than 
\eqref{e:eqprobl}, which involves a finite 
or a countable infinite number of convex constraints. It
will be presented after the following preliminaries.
\begin{notation}
The resolvent of $F\colon C^2\to\RR$ is the set valued
operator
\begin{equation}
\label{e:J_F}
J_F\colon \HH\to 2^C\colon x\mapsto\menge{z\in C}
{(\forall y\in C)\quad F(z,y)+\scal{z-x}{y-z}\geq 0}
\end{equation}
and, for every $\delta\in\RPP$, the $\delta$--resolvent of 
$F\colon C^2\to\RR$ is the set valued
operator
\begin{equation}
\label{e:J_Fdelta}
J_F^{\delta}\colon \HH\to 2^C\colon x\mapsto\menge{z\in C}
{(\forall y\in C)\quad F(z,y)+\scal{z-x}{y-z}\geq -\delta}.
\end{equation}
\begin{lemma}
\label{l:321}
Suppose that $F\colon C^2\to\RR$ satisfies Assumption~\ref{a:1}.
Then the following hold.
\begin{enumerate}
\item\label{l:321i} $\dom J_F=\HH$.
\item\label{l:321ii} $J_F$ is single-valued and firmly nonexpansive.
\item\label{l:321iii} $(\forall x\in\HH)(\forall\delta\in\RPP)
\quad J_Fx\in J_F^{\delta}x$.
\item\label{l:321iv} $(\forall x\in\HH)(\forall\delta\in\RPP)
\quad J_F^{\delta}x\subset B(J_Fx;\sqrt{\delta})$.
\end{enumerate}
\end{lemma}
\begin{proof}
\ref{l:321i}\&\ref{l:321ii}: \cite[Lemma~2.12]{ComH05}.
\ref{l:321iii}: This follows from~\ref{l:321ii}, \eqref{e:J_F},
and \eqref{e:J_Fdelta}. 
\ref{l:321iv}: Fix $x\in\HH$ and $\delta\in\RPP$, and let 
$w\in J_F^{\delta}x$. We deduce from \eqref{e:J_F} and 
\eqref{e:J_Fdelta} that
$F(J_Fx,w)+\scal{J_Fx-x}{w-J_Fx}\geq0$ and
$F(w,J_Fx)+\scal{w-x}{J_Fx-w}\geq-\delta$, respectively.
Adding both inequalities we obtain
$F(w,J_Fx)+F(J_Fx,w)-\|J_Fx-w\|^2\geq-\delta$.
Hence, it follows from Assumption~\ref{a:1ii} that 
$\|J_Fx-w\|^2\leq\delta$, which yields the result.
\end{proof}
\end{notation}

\begin{problem}
\label{prob:eq}
Let $F$ be a function satisfying 
Assumption~\ref{a:1}. Let $(S_i)_{i\in I}$ be a countable (finite or 
countable infinite) family of closed convex subsets of $\HH$ such 
that $S=\cap_{i\in I}S_i\neq\emp$. Let $B\colon\dom B\subset\HH\to\HH$ 
be a monotone and $\chi$--Lipschitz continuous operator such that 
$C\subset\dom B$, and suppose that
\begin{equation}
\label{e:moyendom}
\bigcup_{i\in I}S_i\subset\inte\dom B.
\end{equation}
The problem is to
\begin{equation}
\label{e:probeq}
\text{find}\quad x\in S\quad\text{such that}\quad 
(\forall y\in C)\quad F(x,y)+\scal{Bx}{y-x}\geq 0.
\end{equation}
\end{problem}

Problem~\ref{prob:eq} models a wide variety of problems
including complementarity problems, optimization problems,
feasibility problems, Nash equilibrium problems, variational
inequalities, and fixed point problems 
\cite{Comb01,ComH05,Baus96,Bian96,Blum94,Oett97}.

In the literature, there exist some splitting algorithms
for solving the equilibrium problem
\begin{equation}
\label{e:probeqex}
\text{find}\quad x\in C\quad\text{such that}\quad 
(\forall y\in C)\quad F_1(x,y)+F_2(x,y)\geq 0,
\end{equation}
where $F_1$ and $F_2$ satisfy Assumption~\ref{a:1}. 
These methods take advantage of the properties of $F_1$ and $F_2$ 
separately. For instance, sequential and parallel splitting algorithms 
are proposed in \cite{Moud09}, where the resolvents 
$J_{F_1}$ and $J_{F_2}$ are used. The ergodic convergence to a 
solution to \eqref{e:probeqex} is established without additional 
assumptions. However, when $F_1=F$ and $F_2\colon(x,y)\mapsto
\scal{Bx}{y-x}$ we have $J_{F_2}=J_B=(\Id+B)^{-1}$ 
\cite[Lemma~2.15(i)]{ComH05}, which is often difficult to compute, even in the 
linear case. Moreover, the ergodic method proposed in \cite{Moud09} 
involves vanishing parameters that leads to numerical instabilities, 
which make it of limited use in applications. 
In \cite{ComH05,Moud02} a different approach is developed
to overcome this disadvantage when $B$ is cocoercive. In their
methods, the operator $B$ is computed explicitly and the weakly 
convergence to a solution to \eqref{e:probeq} when $S=C$ 
is demonstrated.

In this section we propose the following non-ergodic algorithm for 
solving the general case considered in Problem~\ref{prob:eq}. This 
approach can deal with errors in the computations of the operators
involved. The convergence of the proposed method is a consequence of 
Theorem~\ref{t:0}.
\begin{algorithm}
\label{algo:eq}
Let $(I_n)_{n\in\NN}$ be a sequence of finite subsets of $I$,
let $(e_{1,n})_{n\in\NN}$ and $(e_{2,n})_{n\in\NN}$ be sequences 
in $\HH$ such that 
$\sum_{n\in\NN}\|e_{1,n}\|<\pinf$ and  
$\sum_{n\in\NN}\|e_{2,n}\|<\pinf$, and let $(\delta_n)_{n\in\NN}$
a sequence in $\RPP$ such that $\sum_{n\in\NN}\sqrt{\delta_n}<\pinf$.
Let $\varepsilon\in\left]0,1/(\chi+1)\right[$, let 
$(\gamma_n)_{n\in\NN}$ be a sequence in 
$\left[\varepsilon,(1-\varepsilon)/\chi\right]$, 
let $\cup_{n\in\NN}\{\omega_{i,n}\}_{i\in I_n}\subset
\left[\varepsilon,1\right]$ be
such that, for every $n\in\NN$, $\sum_{i\in I_n}\omega_{i,n}=1$,
let $x_0\in\dom B$, and let $(x_n)_{n\in\NN}$ be the sequence 
generated by the following routine.
\begin{align}
\label{e:eq1}
(\forall n\in\NN)\quad
&\left\lfloor 
\begin{array}{l}
y_n=x_n-\gamma_n(Bx_n+e_{1,n})\\
q_n\in J_{\gamma_nF}^{\delta_n}y_n\\
r_n=q_n-\gamma_n(Bq_n+e_{2,n})\\
z_n=x_n-y_n+r_n\\
x_{n+1}=\sum_{i\in I_n}
\omega_{i,n}P_{S_i}z_n.
\end{array}
\right.
\end{align}
\end{algorithm}
\begin{remark}
In Algorithm~\ref{algo:eq}, the sequences $(e_{1,n})_{n\in\NN}$ and
$(e_{2,n})_{n\in\NN}$ represent errors in the computation of
the operator $B$. On the other hand, it follows from 
\eqref{e:eq1} and \eqref{e:J_Fdelta} that, for every $n\in\NN$, 
$q_n$ is a solution to
\begin{equation}
\text{find}\quad q\in C\quad\text{such that}\quad 
(\forall y\in C)\quad F(q,y)+\scal{y-y_n}{y-q}\geq-\delta_n.
\end{equation}
Thus, we obtain from \eqref{e:J_F} that $q_n$ can be interpreted as 
an approximate computation of the resolvent $J_{\gamma_nF}y_n$. 
\end{remark}

\begin{proposition}
\label{t:eq}
Suppose that there exist strictly positive integers $(M_i)_{i\in I}$ 
and $N$ such that
\begin{equation}
\label{e:admissib}
(\forall (i,n)\in I\times\NN)\quad i\in \bigcup_{k=n}^{n+M_i-1}I_k
\quad \text{and}\quad 1\leq \card I_n\leq N,
\end{equation}
and that Problem~\ref{prob:eq} admits at least one solution.
Then Algorithm~\ref{algo:eq} generates an infinite orbit 
$(x_n)_{n\in\NN}$ which converges weakly to a solution to 
Problem~\ref{prob:eq}.
\end{proposition}
\begin{proof}
First, let us prove that Problem~\ref{prob:eq} is a particular case
of Problem~\ref{prob:1}. Set
\begin{equation}
\label{e:opeq}
(\forall n\in\NN)\quad \beta_n=\gamma_n\chi,\:\: T_n=J_{\gamma_n F},
\:\:\text{and}\:\: R_n=\Id-\gamma_n B.
\end{equation}

Note that $(\beta_n)_{n\in\NN}$ is a sequence in 
$\left]0,1-\varepsilon\right]$ and, for every $n\in\NN$, 
$T_n$ is firmly nonexpansive \cite[Lemma~2.12]{ComH05} 
and $\Id-R_n=\gamma_nB$ is $\beta_n$--Lipschitz-continuous and 
monotone. Hence, it follows from \cite[Theorem~1]{BroP67} that 
the operators $(R_n)_{n\in\NN}$ are pseudo contractive. 
In addition, we deduce from \eqref{e:J_F} and \eqref{e:opeq} 
that $(\forall n\in\NN)\quad x\in\Fix T_nR_n\,\,
\Leftrightarrow\,\,(\forall n\in\NN)(\forall y\in C)\,\, 
\gamma_nF(x,y)+\scal{x-R_nx}{y-x}\geq0\,\,\Leftrightarrow\,\,
(\forall y\in C)\,\,F(x,y)+\scal{Bx}{y-x}\geq0$. Altogether, 
we deduce that Problem~\ref{prob:eq} is a particular case of 
Problem~\ref{prob:1} and 
\begin{equation}
\label{e:Z22}
Z=S\cap\bigcap_{n\in\NN}\Fix T_nR_n=S\cap\menge{x\in C}
{(\forall y\in C)\:\: 
F(x,y)+\scal{Bx}{y-x}\geq0}\neq\emp. 
\end{equation}

Now let us show that Algorithm~\ref{algo:eq} is deduced from 
Algorithm~\ref{algo:1}. Set
\begin{equation}
\label{e:erreureq}
(\forall n\in\NN)\quad
\begin{cases}
a_n=-\gamma_ne_{1,n}\\
b_n=q_n-J_{\gamma_nF}y_n\\
c_n=-\gamma_ne_{2,n}
\end{cases}
\quad\text{and}\quad
\begin{cases}
\lambda_n=1\\
Q_n=\sum_{i\in I_n}\omega_{i,n}P_{S_i}.
\end{cases}
\end{equation}
Then, since $\sup_{n\in\NN}\gamma_n<\chi^{-1}$, we have 
$\sum_{n\in\NN}\|a_n\|<\pinf$ and $\sum_{n\in\NN}\|c_n\|<\pinf$. 
Moreover, it follows from \eqref{e:eq1} and Lemma~\ref{l:321iv}
that $\sum_{n\in\NN}\|b_n\|<\pinf$, and, for every $x\in\HH$ and 
$n\in\NN$, $Q_nx$ is the projection onto the closed affine 
half-space $H_n(x)=\menge{z\in\HH}{\scal{z-Q_nx}{x-Q_nx}\leq0}$,
which satisfies $S\subset\cap_{i\in I_n}S_i=\Fix Q_n\subset H_n(x)$ 
\cite[Proposition~2.4]{Comb01}. 
On the other hand, we have $x_0\in\dom B$ and it follows from 
\eqref{e:moyendom} and the convexity of $\inte\dom B$ 
\cite[Theorem~27.1]{Simo08} that
\begin{equation}
(\forall n\in\NN)\quad
\ran\Bigg(\sum_{i\in I_n}
\omega_{i,n}P_{S_i}\Bigg)\subset\conv\,\Bigg(\bigcup_{i\in I_n}
S_i\Bigg)
\subset\conv\,\Bigg(\bigcup_{i\in I}S_i\Bigg)
\subset\inte\dom B.
\end{equation}
Hence, we conclude from \eqref{e:eq1} that, for every 
$n\in\NN\setminus\{0\}$, $x_n\in\inte\dom B$. Moreover, for every 
$n\in\NN$, $q_n\in C\subset\dom B$. Altogether, from \eqref{e:opeq} 
and \eqref{e:erreureq}, we deduce that Algorithm~\ref{algo:eq} is a 
particular case of Algorithm~\ref{algo:1}, which generates an 
infinite orbit $(x_n)_{n\in\NN}$.

Finally, let us show that \eqref{e:condconv} holds.
Suppose that $x_{k_n}\weakly x$, $x_n-T_nR_nx_n\to 0$, 
$z_n-x_n\to0$, $z_n-Q_nz_n\to0$, and, for every $n\in\NN$, 
denote by $p_n=T_nR_nx_n$. Hence, $p_{k_n}\weakly x$ and it follows 
from \eqref{e:opeq} and \eqref{e:J_F} that, for every $n\in\NN$,
\begin{align}
\label{e:eqeq2}
p_n=T_nR_nx_n
\:\:&\Leftrightarrow\:\: (\forall z\in C)\:\:\:
F(p_n,z)\!+\!\frac{1}{\gamma_n}\scal{p_n-x_n}{z-p_n}
+\scal{Bx_n}{z-p_n}\geq 0\nonumber\\
&\Leftrightarrow\:\: (\forall z\in C)\:\:\:
F(p_n,z)\!+\!\frac{1}{\gamma_n}\scal{p_n-x_n}{z-p_n}\nonumber\\
&\hspace{4.1cm}+\scal{Bx_n-Bp_n}{z-p_n}+\scal{Bp_n}{z-p_n}\geq 0
\nonumber\\
&\Leftrightarrow\:\: (\forall z\in C)\:\:\:
G(p_n,z)\!+\!\frac{1}{\gamma_n}\scal{p_n-x_n}{z-p_n}
+\scal{Bx_n-Bp_n}{z-p_n}\geq 0,
\end{align}
where
\begin{equation}
\label{e:HHHH}
G\colon C^2\to\RR\colon(x,y)\mapsto F(x,y)+\scal{Bx}{y-x}
\end{equation} 
satisfies the Assumption~\ref{a:1} \cite[Lemma~2.15(i)]{ComH05}.
Moreover, since $\inf_{n\in\NN}\gamma_{k_n}>0$, 
$x_{k_n}-p_{k_n}\to 0$, and $(p_{k_n})_{n\in\NN}$ is
bounded, we have $(\forall z\in C)
\;\scal{p_{k_n}-x_{k_n}}{z-p_{k_n}}/\gamma_{k_n}\to0$,
and from the Lipschitz-continuity of $B$ we obtain
$(\forall z\in C)\;
\scal{Bx_{k_n}-Bp_{k_n}}{z-p_{k_n}}\to0$.
Hence, we deduce from $p_{k_n}\weakly x$, 
Assumption~\ref{a:1}\ref{a:1iii},  
Assumption~\ref{a:1}\ref{a:1ii}, and \eqref{e:eqeq2} that
\begin{align}
\label{e:ineqeq3}
(\forall z\in C)\quad
G(z,x)&\leq\varliminf G(z,p_{k_n})\nonumber\\
&\leq\varliminf-G(p_{k_n},z)\nonumber\\
&\leq \varliminf
\frac{1}{\gamma_{k_n}}\scal{p_{k_n}-x_{k_n}}{z-p_{k_n}}
+\scal{Bx_{k_n}-Bp_{k_n}}{z-p_{k_n}}\nonumber\\
&=0.
\end{align}
Now let $\varepsilon\in\left]0,1\right]$ and $y\in C$. By convexity
of $C$ we have $x_{\varepsilon}=(1-\varepsilon)x
+\varepsilon y\in C$. Thus, Assumption~\ref{a:1}\ref{a:1i}, 
Assumption~\ref{a:1}\ref{a:1iii},
and \eqref{e:ineqeq3} with $z=x_{\varepsilon}$ yield
\begin{equation}
0=G (x_{\varepsilon},x_{\varepsilon})
\leq(1-\varepsilon)G(x_{\varepsilon},x)
+\varepsilon G(x_{\varepsilon},y)
\leq \varepsilon G(x_{\varepsilon},y), 
\end{equation}
whence $G(x_{\varepsilon},y)\geq0$. In view
of Assumption~\ref{a:1}\ref{a:1iv}, we conclude that 
$G(x,y)\geq\varlimsup_{\varepsilon\to0^+} G(x_{\varepsilon},y)\geq0$,
which yields 
\begin{equation}
\label{e:proveI}
(\forall y\in C)\quad G(x,y)=F(x,y)+\scal{Bx}{y-x}\geq 0.
\end{equation}
Now, let us prove that $x\in S$.
Since $z_n-x_n\to0$ and $z_n-Q_nz_n\to0$, \eqref{e:erreureq} yields
\begin{equation}
\label{e:xnplus1mxn}
(\forall n\in\NN)\quad
\|x_{n+1}-x_n\|=\|Q_nz_n-x_n\|\leq\|Q_nz_n-z_n\|+\|z_n-x_n\|
\to0.
\end{equation}
Now, fix $i\in I$. In view of \eqref{e:admissib}, 
there exists a sequence $(j_n)_{n\in\NN}$ in $\NN$ such that,
for every $n\in\NN$, $k_n\leq j_n\leq k_n+M_i-1
\:\:\text{and}\:\:i\in I_{j_n}$.
For every $n\in\NN$, it follows from \eqref{e:xnplus1mxn} that
\begin{equation}
\|x_{j_n}-x_{k_n}\|\leq\sum_{\ell=k_n}^{k_n+M_i-2}
\|x_{\ell+1}-x_{\ell}\|\leq(M_i-1)\max_{k_n\leq\ell
\leq k_n+M_i-2}\|x_{\ell+1}-x_{\ell}\|\to0.
\end{equation}
Thus, we deduce from $x_{k_n}\weakly x$ and $z_{j_n}-x_{j_n}\to0$ 
that $z_{j_n}\weakly x$.
On the other hand, let $z\in S$ and $n\in\NN$. Since, for every 
$\ell\in I_{j_n}$, $P_{S_\ell}z=z$,
and $\Id-P_{S_\ell}$ is firmly nonexpansive, 
from \eqref{e:eq1} and \eqref{e:erreureq} we have
\begin{align}
\|P_{S_i}z_{j_n}-z_{j_n}\|^2
&\leq\max_{\ell\in I_{j_n}}\|P_{S_{\ell}}z_{j_n}-z_{j_n}\|^2\nonumber\\
&\leq \frac{1}{\varepsilon}\sum_{\ell\in I_{j_n}}\omega_{\ell,j_n}
\|P_{S_{\ell}}z_{j_n}-z_{j_n}\|^2\nonumber\\
&\leq\frac{1}{\varepsilon}\sum_{\ell\in I_{j_n}}
\omega_{\ell,j_n}\scal{z-z_{j_n}}{(\Id-P_{S_{\ell}})z-
(\Id-P_{S_{\ell}})z_{j_n}}\nonumber\\
&=\frac{1}{\varepsilon}\Scal{z-z_{j_n}}
{\sum_{\ell\in I_{j_n}}\omega_{\ell,j_n}P_{S_{\ell}}z_{j_n}-z_{j_n}}
\nonumber\\
&\leq \frac{1}{\varepsilon}\|z-z_{j_n}\|\,
\|Q_{j_n}z_{j_n}-z_{j_n}\|.
\end{align}
Hence, since $(z_{j_n})_{n\in\NN}$ is a bounded sequence and 
$Q_{j_n}z_{j_n}-z_{j_n}\to0$, we deduce that 
$P_{S_i}z_{j_n}-z_{j_n}\to0$. The maximally monotonicity of 
$\Id-P_{S_i}$ yields that its graph is sequentially weakly-strongly 
closed, and since $z_{j_n}\weakly x$, we conclude that 
$x=P_{S_i}x\in S_i$. Altogether, from \eqref{e:proveI} and 
\eqref{e:Z22} we deduce that $x\in Z$, and the result follows 
from Theorem~\ref{t:0}.
\end{proof}\\
\vspace{0.3cm}\\
\noindent{\Large{\bf Acknowledgement}}\vspace{0.3cm}\\
I thank Professor Patrick L. Combettes for bringing this problem to my 
attention and for helpful discussions.

\end{document}